\documentclass[review]{elsarticle}
\usepackage{lineno,hyperref}
\usepackage{amsmath,amssymb,bm}
\usepackage{physics}
\usepackage{booktabs}
\usepackage{multirow}
\usepackage{placeins}
\usepackage{amsmath,amssymb,bm,amsthm}
\usepackage{multirow}
\usepackage{mathtools}
\usepackage[dvipsnames]{xcolor}
\usepackage{graphicx}

\usepackage{dcolumn}
\usepackage{latexsym}

\usepackage{comment,ulem} 
\usepackage[nottoc,notlot,notlof]{tocbibind}
\usepackage[title,titletoc]{appendix}
\usepackage{hyperref} 
\usepackage{ulem}

\usepackage{subcaption}
\usepackage{tabularx}

\def\<{\langle}
\def\>{\rangle}

\newcommand{\SmallMatrix}[1]{{%
  \tiny\arraycolsep=0.3\arraycolsep\ensuremath{\begin{pmatrix}#1\end{pmatrix}}}}
\newcommand{\rk}{\mathrm{rank}}  
\renewcommand{\tr}{\mathrm{tr}} 

\newcommand{\Mnc}{\mathop{M_n}\left( \mathbb{C}\right)} 
\newcommand{\Mnsc}{\mathop{M_{n^2}\left( \mathbb{C} \right)}} 
 
\newcommand{\I}{\mathop{\mathbb{I}}\nolimits}
\newcommand{\R}{\mathop{\mathbb{R}}\nolimits}
\newcommand{\C}{\mathop{\mathbb{C}}\nolimits}
\renewcommand{\ketbra}[2]{| #1 \rangle \langle #2 | }

\renewcommand{\ket}[1]{| #1 \rangle} 

\newcommand{\cc}[1]{\overline{#1}} 
\renewcommand{\Re}{\mathop{\mathrm{Re}}} 
\newcommand{\f}[1]{\left| \left| \ #1 \ \right|\right|_{\mathbf{\mathrm{F}}}} 
\newcommand{\ft}[1]{\left| \left| \ #1 \  \right| \right|_{(2),2}} 
\renewcommand{\norm}[1]{\left| \left| #1 \right| \right|} 
\renewcommand{\op}[1]{\left| \left| #1 \right| \right|_{\mathrm{op}}} 
\let\sup\relax
\DeclareMathOperator*{\sup}{sup}

\newtheorem{Prop}{Proposition}

\newtheorem{Remark}{\it Remark}

\newtheorem{Conj}{Conjecture}

\newcommand{\Ah}{s_1}
\newcommand{\Ash}{s_2}

\begin{document}

\begin{frontmatter}

\title{Bounding the Frobenius norm of a $q$-deformed commutator}

\author[1]{Dariusz Chru\'sci\'nski}
\ead{darch@fizyka.umk.pl}
\author[2]{Gen Kimura}
\ead{gen@shibaura-it.ac.jp}
\author[3]{Hiromichi Ohno}
\ead{h\_ohno@shinshu-u.ac.jp}
\author[1]{Tanmay Singal \corref{C}}
\ead{tanmaysingal@gmail.com}

\address[1]{Institute of Physics, Faculty of Physics, Astronomy and Informatics  Nicolaus Copernicus University, Grudzi\c{a}dzka 5/7, 87--100 Toru\'n, Poland}
\address[2]{College of Systems Engineering and Science,
Shibaura Institute of Technology, Saitama 330-8570, Japan}
\address[3]{Department of Mathematics, Faculty of Engineering, Shinshu University,
4-17-1 Wakasato, Nagano 380-8553, Japan.}
\cortext[C]{Corresponding author}

\date{}
\begin{abstract} \noindent
For two $n \times n$ complex matrices $A$ and $B$, we define the $q$-deformed commutator as $[ \ A, B \ ]_q \ \coloneqq \ A B - q BA$ for a real parameter $q$. In this paper, we investigate a generalization of the B\"{o}ttcher-Wenzel inequality which gives the sharp upper bound of the (Frobenius) norm of the commutator. In our generalisation, we investigate sharp upper bounds on the $q$-deformed commutator. This generalization can be studied in two different scenarios: firstly bounds for general matrices, and secondly for traceless matrices. For both scenarios, partial answers and conjectures are given for positive and negative $q$. In particular, denoting the Frobenius norm by $\f{.}$, when either $A$ or $B$ is normal, we prove the following inequality to be true and sharp:  $\f{[ \ A , B \ ]_q}^2 \le \left(1+q^2 \right) \f{A}^{{2}}\f{B}^{{2}}$ for positive $q$. Also, we conjecture that the same bound is true for positive $q$ when either $A$ or $B$ is traceless. For negative $q$, we  conjecture other sharp upper bounds to be true for the generic scenarios and the scenario when either of $A$ or $B$ is traceless. All conjectures are supported with numerics and proved for $n=2$. 
\end{abstract}

\end{frontmatter}

\section{Introduction}

In \cite{BW2005}, B\"{o}ttcher and Wenzel conjectured the following sharp upper bound to the Frobenius norm of the commutator of two $n \times n$ complex matrices $A$ and $B$:
\begin{equation}
    \label{eq:BW}
    \f{ \ AB \ - \ BA \ }^2 \ \le \ 2 \f{A}^2 \f{B}^2. 
\end{equation}
In subsequent years many different proofs of the conjecture were given: L\'{a}sl\'{o} proved the case for $3 \times 3$ real matrices \cite{Laslo}, Lu \cite{Lu1} and, independently, Wong and Jin \cite{Jin} proved it for $n \times n$ real matrices, after which the complex case was proved by B\"{o}ttcher and Wenzel in \cite{BW2008}. In \cite{Audenaert}, Audenaert gave a simple proof of the conjecture. In particular, Audenaert proved a stronger bound using the Ky Fan $(2),2$ norm. The Ky Fan $(2),2$ norm of $A$ is defined as \cite{Horn} $\ft{A} \coloneqq \sqrt{ \Ah^2 + \Ash^2 }$, where $\Ah$ and $\Ash$ are the largest and second largest singular values of $A$. See also \cite{Wenzel} and \cite{Lu2} for other proofs of the original 
{B\"ottcher-Wenzel} (BW) inequality. Since the proof of the original inequality, the result has been generalised using other norms (notably among them, the Schatten-$p$ norm) of the commutators and other cases \cite{WA, Wu, Fong, Cheng5, Cheng4, Cheng3, Gil, Xie, Kan, LW, Ge1,Cheng2, Cheng1, Liu}. See \cite{Cheng4} and \cite{LW} for a review of the history of this problem. 

In this work we deform the commutator of $A, B \in M_n(\mathbb{C})$,
 \begin{equation}
    \label{eq:q}
       [ \ A, \ B \ ]_q  \ \coloneqq \ AB - q BA , \ \ \ q \in \mathbb{R} ,
\end{equation} 
and search for the sharp upper bound on the Frobenius norm of $  [ \ A, \ B \ ]_q$: 
\begin{equation}\label{eq:genBW1}
\f{[ \ A, \ B \ ]_q}^2 \le c(n) \f{A}^2 \f{B}^2, 
\end{equation}
where the optimal bound $c(n)$ can be characterized by 
\begin{equation}\label{eq:optBdd}
c(n) = \sup_{A,B (\neq 0) \in M_n(\mathbb{C})} \frac{\f{[ \ A, \ B \ ]_q}^2}{\f{A}^2 \f{B}^2}. 
\end{equation}

While the standard commutator is ubiquitous in physics and mathematics, deformation of the standard commutation relations have also been studied for quite sometime. With pioneering works of E. K. Sklyanin \cite{Sklyanin82} and independently by P. P. Kulish and Reshetikhin \cite{Kulish_Reshetikhin} on $SU_q(n)$ ($q$-deformed $SU(n)$), the topic of quantum groups gained significant {traction} after the works of Jimbo \cite{Jimbo}, Drinfeld \cite{Drinfeld} and Woronowicz \cite{Woronowicz87,Woronowicz89}. The topic continues to draw interest, for instance see \cite{Franco_Fagnola} for a recent work on it.
For a physics-friendly introduction to quantum groups we refer the reader to \cite{RJaggi}, whereas to the more mathematically inclined reader we refer \cite{Majid95}. 

Taking inspiration from both, the large body of work on $q$-deformed commutators and also {the} BW conjecture, we were motivated to study the corresponding upper bound  {on} the Frobenius norm of the $q$-deformed commutator. Recently, we found an application of BW inequality to analyze the relaxation rates for the evolution of open quantum systems \cite{RR1,RR2}. 

In generalizing BW inequality, we study two different scenarios: 

[Case 1] The sharp bound on \eqref{eq:genBW1} for when $A$ and $B$ are arbitrary $n \times n$ matrices.

[Case 2] The sharp bound when $A$ or $B$ in \eqref{eq:genBW1} are traceless matrices. Note here that, in the original BW inequality, only the traceless parts of $A$ and $B$ are essential.  \newline

\noindent Although this project is not yet complete, we find several interesting partial answers and we put forth several conjectures which we support by numerics. These conjectures are given below.

[Case 1-1] For negative $q \le 0$, we have the sharp bound: 
\begin{equation}\label{eq:Neg}
\f{[A,B]_q}^2\leq\left(1-q\right)^2 \f{A}^2\f{B}^2. 
\end{equation}
In particular, the sharp bound for anti-commutator ($q=-1$) is given by 
\begin{equation}
\f{AB +BA}^2\leq 4 \f{A}^2\f{B}^2. 
\end{equation}
Note that while Eq. \eqref{eq:Neg} is trivial, it is sharp. 

[Case 1-2] For $q>0$, while we weren't able to find the optimal sharp bound, we prove that, if either $A$ or $B$ is normal, the sharp bound is given by  
\begin{equation}
\label{eq:q_commutator_BW}
\f{[A,B]_q}^2\leq\left(1+q^2 \right)\f{A}^2\f{B}^2.
\end{equation}  
Indeed, we prove the following tighter inequality: if $A$ is normal, 
\begin{equation}
\label{eq:q_commutator_1}
\f{[A,B]_q}^2\leq\left(1+q^2 \right)\ft{A}^2\f{B}^2. 
\end{equation}  
We find that, however \eqref{eq:q_commutator_BW} can be  violated if neither $A$ nor $B$ are normal. 

We find that imposing tracelessness on either $A$ or $B$ changes the sharp upper bounds compared to the generic case.

[Case 2-1] For $q \le 0$, we conjecture that the sharp bound depends on $n$ the following way:
\begin{equation*}\label{eq:TLBdd_introd}
{\f{[A,B]_q}^2}  \ \leq \ \max[g(n)(1-q)^2, 1+q^2]  \ \f{A}^2 \ \f{B}^2,
\end{equation*} 
where \[g(n)  \ =  \ \frac{n^2-3n+3}{n(n-1)}.  \] When $n=2,3$, this agrees with the bound in Eq. \eqref{eq:genBW1}, but this isn't true for $n\ge4$.

[Case 2-2] For $q > 0$: we conjecture that Eq. \eqref{eq:q_commutator_BW} gives the sharp upper bound. 

All these conjectures are supported by numerics and are proved for $n=2$. 

\bigskip 


\subsection*{Notation}
For $a \in \C$, $\cc{a}$ denotes the complex conjugate of $a$. 
For a complex matrix $A = \left( a_{ij} \right) \in M_n \left( \C \right)$, we denote by $A^\dagger, A^T$ and $\tr (A)$, the adjoint (conjugate transpose), the transpose and the trace of $A$ respectively.  
In the following, we use three matrix norms for $A \in M_n \left( \C \right)$, the operator norm (or spectral norm) $\op{A}$, the Frobenius norm $\f{A}= \sqrt{\tr{AA^\dag}}$, and the Ky Fan $(2),2$ norm $\ft{A}$ \cite{Horn}
. Denoting the singular values of $A$ by $s_1 \ge s_2 \ge \cdots \ge s_n $ arranged in descending order, one has $\op{A}=\Ah$, $\f{A}= \sqrt{\sum_{i=1}^n s_i^2}$, and $\ft{A} = \sqrt{ \Ah^2 \ + \ \Ash^2 }$. Thus, $\op{A} \le \ft{A} \le \f{A}$. Any vector in $\R^3$ will be denoted by bold letters $\textbf{r}$, $\textbf{s}$, etc. $\sigma_1$, $\sigma_2$ and $\sigma_3$ denote the well-known $2 \times 2$ Pauli matrices. For $\textbf{r} \in \R^3$, $\textbf{r}\cdot\boldsymbol{\sigma}$ denotes $\sum_{i=1}^3 r_i \sigma_i$. At times, we shall use the Dirac notation: A complex vector $(x_i)_{i=1}^m$ in $\C^m$ is denoted by $\ket{x}$ where $\langle x| y \rangle = \sum_i \cc{x_{i}}y_i $ is the inner product while $|x\rangle \langle y |$ is the matrix with entries $(x_i \cc{y_j})$. The vector norm is given and denoted by $\norm{\ \ket{x} \ } = \sqrt{\langle x| x \rangle}$. We also use the shorthand notation $|x\rangle | y \rangle$ for the tensor product of $|x\rangle$ and $| y \rangle$. 

\section{General bounds}\label{sec:GBdd}


In this section, we consider the general bound \eqref{eq:genBW1}. 
Note when $q \le 0$ and $A$ and $B$ are arbitrary,  the bound \eqref{eq:Neg} can beeasily verified by applying the triangle inequality and the submultiplicativity of the Frobenius norm $\f{ AB } \le \f{A} \f{B}$:
\begin{equation}
\label{eq:minus_q}
\f{ [A,B]_q } \le \f{ AB } -q \f{ BA } \le (1-q) \f{A} \f{B} .    
\end{equation}
Note that this upper bound is attained when $A=B$ is a $\mathrm{rank}$ one orthogonal projector, hence 
the bound \eqref{eq:Neg} is sharp. 

For $q >0 $, we first consider the case when either $A$ or $B$ is normal.
\begin{Prop}
\label{prop:1}
Let $A$ be normal in $\Mnc$. Then for $q \ge 0$
\begin{equation*}
{\f{[A,B]_q}^2}  \ \leq \ \left( 1+q^2 \right) \ \ft{A}^2 \ \f{B}^2,
\end{equation*}
where the bound is sharp.   
\end{Prop}
\begin{proof}  Let $A$ be a normal matrix. Since $\f{ [A,B]_q}^2$ is unitarily invariant we can choose a basis where $A$ is a diagonal matrix $A = {\rm diag}[a_1,\ldots,a_n]$. Let $(b_{ij})$ be the corresponding matrix elements of $B$. A direct computation shows 
$$
\f{[A,B]_q}^2 = \sum_{ij} |b_{ij}|^2 \left(|a_i|^2 + q^2 |a_j|^2 - 2q \Re \overline{a_j}a_i \right) .
$$
Since ${|a_j + q a_i|^2 \geq 0 }$, the following is true for $q\ge0$:

$$ -2q \Re \overline{a_j}a_i \  \le \  \left(1-\delta_{ij}\right) \left(  |a_j|^2 + q^2 |a_i|^2 \right), $$ 
and hence
\begin{align*}  |a_i|^2  +   q^2 |a_j|^2  -   2q \Re \overline{a_j}a_i \   \le    \ \left(   1+q^2  \right)\left(   | a_i|^2 + \left(1-\delta_{ij}\right) |a_j|^2   \right) \ \le  \ \left(1+q^2 \right)  \ft{A}^2,  \end{align*} which proves the result as $\f{B}^2 =  \sum_{ij} |b_{ij}|^2$. The bound is attained e.g. by matrices $A = {\rm diag}[1,-q,0,\ldots,0]$ and $B = (b_{ij})$ with $b_{ij} = 0$ except for $b_{12} = 1$. 
\end{proof}

Clearly, if $B$ is normal, then 

\begin{equation*}
 { \f{[A,B]_q}^2} \ \leq \ \left( 1+q^2 \right) \ \f{A}^2 \ \ft{B}^2 \ 
\end{equation*}
Hence, one obtains: 

\begin{Prop}
For  $q > 0$, if either $A$ or $B$ is normal, 
\begin{equation*}
{\f{[A,B]_q}^2}  \ \leq \ \left( 1+q^2 \right) \ \f{A}^2 \ \f{B}^2,  
\end{equation*}
where the bound is sharp.   
\end{Prop}

Setting $q=1$ in the above inequality, we recover the BW inequality, which is true even when neither $A$ nor $B$ are normal. However, the bound in Eq.~\eqref{eq:q_commutator_BW} can be violated when $A$ and $B$ are generic when $q \neq 1$ but greater than $0$. A simple counter example for $q=2$ is given e.g. by $A = 
\begin{pmatrix}
 2 & 8 \\
 0 & -1 \\
\end{pmatrix}, B = 
\begin{pmatrix}
 2 & 0 \\
 -8 & -1 \\
\end{pmatrix}
$ a direct computation gives $\f{A}^2 = \f{B}^2 = 69$, and $\| [A,B]_2 \|^2 = 23953$, which is strictly greater than the bound $(1+2^2) \f{A}^2 \f{B}^2 = 23805$. One may elaborate this example as follows: 
 let
\begin{equation}
A = \begin{pmatrix} q & \sqrt{t} \\ 0 & -1 \end{pmatrix}, \quad 
B = \begin{pmatrix} q & 0 \\ -\sqrt{t} & -1 \end{pmatrix},
\end{equation}
\noindent where $t\ge 0$. One has
\[
[A, B]_q 
= \begin{pmatrix} q^2(1-q)-t & -(1+q^2) \sqrt{t} \\ (1+q^2) \sqrt{t} & qt +1-q \end{pmatrix}
\]
\noindent and hence
\[
\| [A,B]_q \|^2 = t^2 (1+q^2) +2t(1+q)(1+q^3) + (1-q)^2(1+q^4).
\]
\noindent Define a function $f(t)$ by
\[
f(t) = \frac{\f{ [A,B]_q }^2}{\f{A}^2 \f{B}^2}
=\frac{t^2 (1+q^2) +2t(1+q)(1+q^3) + (1-q)^2(1+q^4)}{(t+q^2+1)^2}.
\]
We remark, on the side, that when $t \to \infty$, we obtain $f(t) \to 1+q^2$.
\noindent The numerator $f'(t)$ is  {a quadratic polynomial in $t$ of the following form:}
 {\[-2q\left(1-q\right)^2\left(t+q^2+1\right)\left(t-t_\mathrm{max}\right), \ \mathrm{where} \  t_\mathrm{max} \ = \ \frac{3q^4 + 2q^2 +3}{(1-q)^2}.
\] }
\noindent  {For $t\ge 0 $}, $f$ acquires its  {maximum} when $t=  t_\mathrm{max}$, and $f(t_\mathrm{max})$ is given by 
\noindent,
\begin{equation}\label{eq:fmax}
f\left(t_\mathrm{max}\right) \ = \ \frac{1 - h(q)x(q) }{1-x(q)} \ \left(1+q^2\right),
\end{equation}
where \[ x(q) \ :=  \ \dfrac{q\left(1-q\right)^2}{2\left(1+q^4\right)}, \ \mathrm{and} \ h(q) \ := \ \dfrac{\left(1+q\right)^2}{2\left(1+q^2\right)}.\]
When $q >0$ and $q \neq 1$, $0 < x(q),h(q) < 1$. Thus $\frac{1-h(q)x(q)}{1-x(q)} > 1$. Thus $f(t_\mathrm{max}) > \left( 1+q^2 \right)$. Furthermore,  $f'(t) < 0$ for $t > t_\mathrm{max}$, and since $f(t) \to 1+q^2$ when $t \to \infty$, $f(t) > 1+q^2$ for $t \ge t_{\mathrm{max}}$.

In order to find the best bound, we performed a numerical optimization for Eq. \eqref{eq:optBdd}, i.e., we numerically maximized $\f{ [A,B]_q}^2/\left(\f{A}^2 \f{B}^2\right)$. Fig.~\ref{fig:genBdd} shows the plots for $n=2,3,4$. 
For negative $q$, the numerical maximum (denoted by red points) lies on the curve $(1-q)^2$ (plotted with the dotted line), and this observation agrees with the proof in Eq. \eqref{eq:minus_q}. However, for positive $q$, one observes that the numerical maximum slightly exceeds the bound in Eq. \eqref{eq:q_commutator_BW} (plotted with the solid line) (a magnification has been attached to the figure for $n=2$ to enable the reader to see this clearly; note in passing that while $f(t_{\max})$ in Eq. \eqref{eq:fmax} exceeds the bound, it still doesn't achieve the numerical maximum.). We performed numerical maximzation of the bound till $n=10$, and in these optimizations, the numerics suggests that the upper bound doesn't depend on $n$.

\begin{figure}[th]
\includegraphics[width=\textwidth]{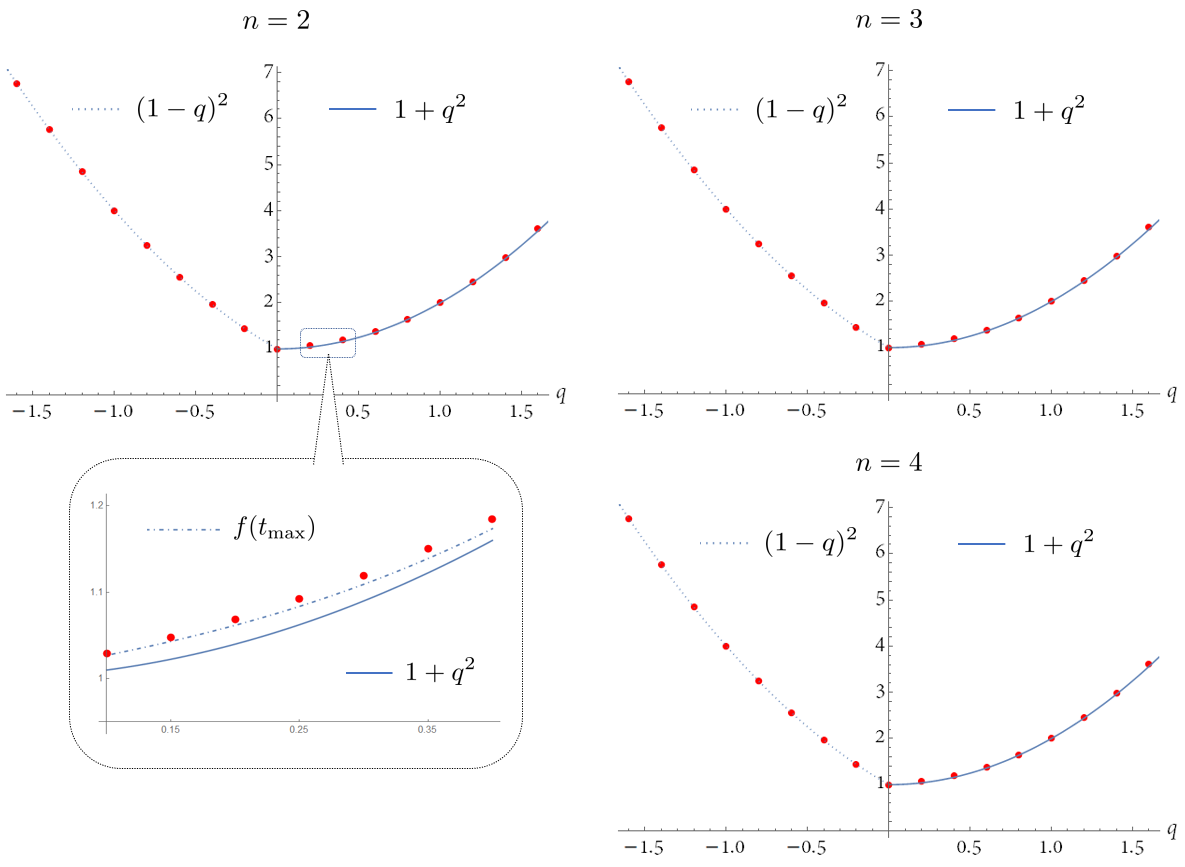}
\caption{
Numerical optimization of the bound for $n=2,3,4$. Reds points are plots of the numerical maximization of $\f{ [A,B]_q}^2/\left(\f{A}^2 \f{B}^2\right)$. 
}\label{fig:genBdd}
\end{figure}

\section{Bound for traceless matrices}\label{sec:TLBdd}

In this section, we consider the bound  \eqref{eq:genBW1} when either $A$ or $B$ is constrained to be traceless. For positive $q$, we have the following conjecture: 
\begin{Conj}
\label{CON}\label{con:TLpos}
For any $q > 0$, if $A$ or $B$ is traceless, the inequality \eqref{eq:q_commutator_BW} holds and is sharp.   
\end{Conj}
\noindent We notice again that this conjecture, if true, gives a generalization of BW-inequality since the standard commutator of $A$ and $B$ doesn't depend on their respective traces. 

In the following, we provide a proof of Conjecture \ref{CON} for $n=2$ for any real $q$.
\begin{Prop}
\label{prop:n_2}
For $n=2$ and $\tr{A}=0$, the bound in Eq.~ \eqref{eq:q_commutator_BW} is satisfied and is sharp. 
\end{Prop} ~ \newline
\noindent We give two different proofs of Proposition \ref{prop:n_2}. The first proof is customised for the $n=2$ case, and is quite elementary (as one may anticipate). Contrasted against this, the second proof is based on a generic method which is suggestive for proving Eq. ~ \eqref{eq:q_commutator_BW} for more generic cases when $\tr{A}=0$ or $\tr{B}=0$. \newline \noindent  \begin{proof}
\noindent Note that, as is easily shown, a traceless $2 \times 2$ matrix $A$ is unitarily equivalent to a matrix whose diagonal entries are zero  (see e.g., Example 2.2.3 in \cite{Horn2}). Thus, we can assume that $A = 
\SmallMatrix{
0 & a_{12} \\
a_{21} & 0}
$. Let $B = \SmallMatrix{
b_{11} & b_{12} \\
b_{21} & b_{22}}$ in this basis. Then one has 
\begin{eqnarray*}
X &:=& (1+q^2)\f{A}^2\f{B}^2- \f{[A,B]_q}^2 \\
&=& (1+q^2) |a_{12}|^2 |b_{12}|^2 + (1+q^2) |a_{21}|^2 |b_{21}|^2 \\
&& {}  + |a_{12}|^2 (|b_{11}|^2 + q^2 |b_{22}|^2) + |a_{21}|^2 (|b_{22}|^2 + q^2 |b_{11}|^2) \\ 
&& {} + q (|a_{12}|^2+|a_{21}|^2)(b_{11} \cc{b_{22}}+\cc{b_{11}} b_{22}) + 2q(a_{12} \cc{a_{21}} b_{21} \cc{b_{12}} + \cc{a_{12}} a_{21} \cc{b_{21}} b_{12}) \\
&=&  (1\mp q)^2(|a_{12}|^2 |b_{12}|^2+ |a_{21}|^2 |b_{21}|^2) + |a_{12}|^2 |b_{11}+q b_{22}|^2  \\ 
&& {} + |a_{21}|^2 |b_{22}+q b_{11}|^2 \pm 2 q |a_{12} \cc{b_{12}} \pm a_{21} \cc{b_{21}}|^2.
\end{eqnarray*}
This shows $X \ge 0$ for any $q \in {\mathbb R}$. In the example in the proof of Proposition 1, $\tr{B}=0$. Thus it also serves as a proof of the sharpness for the bound in the traceless case.
\end{proof}  ~ \newline Before giving the second proof, let us recall the well known vectorization technique \cite{Laslo,Watrous,Gilchrist}: given $B \in \Mnc$, one defines a vector $|B \> \!\> \in \mathbb{C}^n \otimes \mathbb{C}^n$ via 
\begin{equation}
    \label{eq:def:vec}
B = \sum_{i,j=1}^n \ b_{ij} \ \ketbra{i}{j} \longrightarrow     |B\>\!\> \ = \ \sum_{i,j=1}^n \ b_{ij}  \ {\ket{j}\ket{i}}, \ 
\end{equation}  \noindent Note that $\f{B}^{2} = \<\!\<B|B\>\!\>$. The transformation $B \longrightarrow AB-qBA$ is linear in $B$ and hence there exists $M_A \in \Mnsc$ which implements this transformation in $\C^n \otimes \C^n$, i.e., 
\begin{equation}
    \label{eq:M_0} 
    |AB - q BA\>\!\> \ = \ M_A \ |B\>\!\>,
\end{equation}
where (see \cite{Laslo,Watrous,Gilchrist}) \begin{equation}
    \label{eq:M}
    M_A \ = \ \I \otimes A \ - q \ A^T \otimes \I.
\end{equation} \noindent Hence inequality  \eqref{eq:q_commutator_BW} is equivalent to 
\begin{equation}
    \label{eq:q_commutator_2}
     \left( \  \dfrac{\norm{  \ M_A   |B\>\!\> \  }}{ \norm{ \ |B\>\!\> \  }}  \   \right)^2 \ \le \ \left(1 + q^2 \right) \ \f{A}^2, 
\end{equation}  \noindent Since the above inequality has to be satisfied for any $B$ it implies 

\begin{equation}
    \label{eq:q_commutator_3}
  \op{M_A}^2 \ \leq  \ \left( 1+q^2\right) \ {\f{A}^2}.
\end{equation}
\begin{proof} 
As mentioned in the first proof, one may choose a basis in which the diagonal elements of $A$ are identically $0$:
\begin{equation}
    \label{eq:n_2_A}
    A \ = \ \begin{pmatrix}
    0 & a \\ b & 0 
    \end{pmatrix}.
\end{equation} \noindent The matrix $M_A M_A^\dag $ (see Eq. \eqref{eq:M}) takes the following form. $M_A M_A^\dag = M_1 \oplus M_2$, where 
\begin{align}\label{eq:M_1}  M_1 =  \begin{pmatrix}|a|^2 + q^2 |b|^2& -q (|a|^2 + |b|^2)\\-q(|a|^2 + |b|^2)&|b|^2 + q^2 |a|^2\end{pmatrix},\end{align} and \begin{align} \label{eq:M_2}   M_2 \ = \begin{pmatrix} \left( 1 + q^2 \right)|b|^2  & -2q\cc{a}b  \\ 
 -2qa\cc{b} & \left( 1+ q^2 \right)|a|^2  \end{pmatrix}.\end{align} \noindent The eigenvalues of $M_i$ are given by $\frac{1}{2} \left( \tr{M_i} \pm \sqrt{\Delta_i} \right)$, where the discriminant $\Delta_i$ is given by, $\Delta_i = \left(\tr{M_i}\right)^2-4 \ \mathrm{det}M_i$. If  $ \left| \ 4  \  \mathrm{det}M_i \ \right| \le \left(\tr{M_i}\right)^2 $, then both eigenvalues of $M_i$ are smaller than $\tr{M_i}$. For both $i=1,2$, we see that $\tr{M_i}= \left(1+q^2\right)\left(|a|^2+|b|^2\right)$, and $\mathrm{det}M_i=(1-q^2)^2|a b|^2$. Since $2 |ab| \le |a|^2 + |b|^2 $, we have that $ \left| \ 4  \  \mathrm{det}M_i \ \right| \le \left(\tr{M_i}\right)^2$. Noting that $\f{A}^2=|a|^2 + |b|^2$, the inequality is proved.
\end{proof}

\begin{Remark} Note that if both $A$ and $B$ are traceless, then \eqref{eq:q_commutator_BW} follows {from} the properties of the Lie algebra $su(2)$. Indeed, for $A = \boldsymbol{a}\cdot \boldsymbol\sigma$ and $B = \boldsymbol{b}\cdot \boldsymbol\sigma$ with $\boldsymbol{a},\boldsymbol{b} \in \C^3$ one easily computes \begin{equation}\label{}
  [A,B]_q =(1-q) \boldsymbol{a} {\cdot} \boldsymbol{b}\, \I + i (1+q) (\boldsymbol{a} \times \boldsymbol{b}) \cdot \boldsymbol\sigma  ,
\end{equation}
and hence \begin{equation}\label{}
 \f{ [A,B]_q }^2 
   = 2 \Big( (1-q)^2 |\boldsymbol{a} {\cdot} \boldsymbol{b}|^2 + (1+q)^2 | \boldsymbol{a} \times \boldsymbol{b}|^2 \Big) ,
\end{equation} \noindent where we have used  

$$ {\rm Tr}[ ( \boldsymbol{a} \cdot \boldsymbol\sigma)(\boldsymbol{b} \cdot \boldsymbol\sigma )] =2 \boldsymbol{a} {\cdot} \boldsymbol{b}.$$ \noindent Hence using

$$  |\boldsymbol{a}  {\cdot} \boldsymbol{b}|^2 \leq |\boldsymbol{a}|^2  |\boldsymbol{b}|^2 \ , \ \ \  |\boldsymbol{a} \times \boldsymbol{b}|^2 \leq |\boldsymbol{a}|^2  |\boldsymbol{b}|^2 , $$ \noindent one obtains \begin{eqnarray}\label{}
 \f{ [A,B]_q }^2 \leq  4 (1+q^2) | \boldsymbol{a}|^2 |\boldsymbol{b}|^2 = (1+q^2) \f{A}^2 \f{B}^2 . 
\end{eqnarray}
\end{Remark}

\begin{Remark}
Note that in the special case when $\tr{A}=0$ and $\rk A = 1$, then the conjecture is satisfied.  Indeed, let $\left\{ \ket{e_i} \right\}_{i=1}^n$ be the standard basis for $\C^n$. When $\tr{A}=0$ and $\rk A = 1$, we can employ a unitary similarity transformation, so that its singular value decomposition is then given in the following form: $A  = a  \ketbra{e_1}{e_2} $, where $a \ge 0$. Hence $A$ is now a traceless matrix, and its non-zero matrix elements are limited to the upper left block of size $2 \times 2$. By following the same steps using the proof of Proposition \ref{prop:n_2}, the result is proved.
\end{Remark}


When $n \ge 3$, Conjecture \ref{CON} is supported with numerical analysis: assuming $\tr{A}=0$, we numerically maximized $\f{ [A,B]_q}^2/\left(\f{A}^2 \f{B}^2\right)$. The plots are shown in Fig.~\ref{fig_TLBdd} for $n=2,3,4,5$. As can be seen the maximum values of $\f{ [A,B]_q}^2/\left(\f{A}^2 \f{B}^2\right)$ lie on the curve $1+q^2$. We performed numerical optimization for the cases upto $n=10$, and these numerics agree with the bound put forth by Conjecture \ref{CON}. 


\begin{figure}[!t]
\centering
 \includegraphics[width=\textwidth]{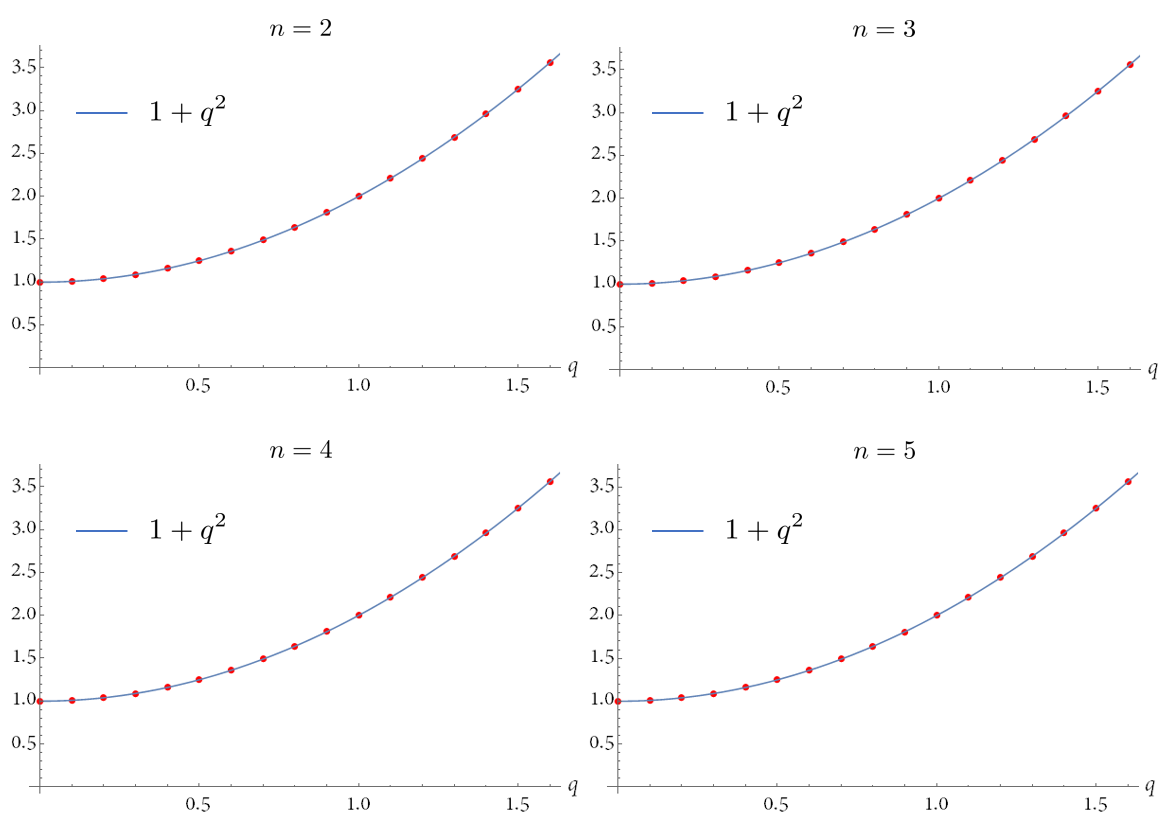}
\caption{Numerical optimization of the bound for $n=2,3,4,5$ for positive $q$: we numerically maximized $\f{ [A,B]_q}^2/\left(\f{A}^2 \f{B}^2\right)$ for traceless $A$. It is seen that the maximum (red points) lies on the curve $1+q^2$ (solid line). 
}\label{fig_TLBdd}
\end{figure}

When $q<0$, it seems that the bound will depend on $n$. For $n=2$, the sharp bound is given by \eqref{eq:q_commutator_BW} as is proved above (note that the signature of $q$ is inessential in the proof of Conjecture \ref{CON} for $n=2$). However, for general $n$, we have the following counter examples. 
 
Let $A$ and $B$ be $n \times n$ trace-less diagonal matrices defined by
\[
A = B = {\rm diag} [n-1, -1, \cdots, -1].
\]
Then,
\[
\|A\|^2 =\|B\|^2= (n-1)^2 +(-1)^2 (n-1) = n(n-1).
\]
 {Also note that } 
\[
AB = {\rm diag} [(n-1)^2, 1, \cdots, 1],
\]
and therefore,
\[
\f{[A,B]_q}^2 = (1-q)^2\f{AB}^2 = \left(1-q\right)^2 ((n-1)^4 + (n-1)) = \left(1-q\right)^2 n(n-1)(n^2 -3n +3).
\]
This implies 
\begin{equation}
    \label{eq:rem_3_1}
\frac{\f{[A,B]_q}^2}{\f{A}^2 \f{B}^2} = g(n) \left(1-q\right)^2,
\end{equation}
where 
\begin{equation}\label{eq:gn}
g(n) = \frac{n^2-3n+3}{n(n-1)}. 
\end{equation}

For $n=2,3$, $g(n) =\frac{1}{2}$. Since $\frac{\left(1-q\right)^2}{2}  \le  \left(1+q^2\right)$ for negative $q$, inequality \eqref{eq:q_commutator_BW} isn't violated. For $n\ge4$, $g(n)>\frac{1}{2}$,
and hence there always exist values of $q<0$, for which inequality \eqref{eq:q_commutator_BW} is violated. Specifically, let 
\begin{subequations}\label{eq:qminmax}
\begin{eqnarray}
q_{\min}(n) &=& \frac{-n^2 (k(n)+1) + n (k(n)+3) -3}{2 n-3}, \\
q_{\max}(n) &=& \frac{n^2 (k(n)-1)- n(k(n)-3)-3}{2 n-3}, 
\end{eqnarray}
\end{subequations}
where
$$
k(n) = \sqrt{\frac{(n-2)(n-3)}{n(n-1)}}. 
$$
Then, for any $q \in [q_{\min}, q_{\max}]$, inequality \eqref{eq:q_commutator_BW} is violated.

With this counter example in mind, we conjecture:
\begin{Conj}\label{CON:Trnq}
For any $q\le 0$, if $A$ or $B$ is traceless, the sharp bound is 
\begin{equation}\label{eq:TLBdd}
{\f{[A,B]_q}^2}  \ \leq \ \max[g(n)(1-q)^2, 1+q^2]  \ \f{A}^2 \ \f{B}^2,
\end{equation}
where $g(n)$ is given by \eqref{eq:gn}.  
\end{Conj} 
As is already mentioned above, 
$g(2)=g(3) = \frac{1}{2}$, and $\frac{\left(1-q\right)^2}{2}  \le  \left(1+q^2\right)$ for negative $q$,
hence for $n=2,3$, the bound is the same as the one for general case. However, for $n \ge 4$, there is the range $[q_{\min}, q_{\max}]$ with \eqref{eq:qminmax}, where the bound is given by
$g(n)(1-q)^2 \f{A}^2 \ \f{B}^2$.

Finally, to support this conjecture, we conducted numerical optimization for $n=2,3,4$.  In Fig.~\ref{fig:Tl23neg}, the numerical maximum (red points) lies on the $1+q^2$ curve for both cases $n=2,3$, while in Fig.~\ref{fig:Tl4neg}, one observes that the maximum (red points) lies on the curve $g(4)(1-q)^2$ if $q \in [q_{\min},q_{\max}]$ and on the curve $1+q^2$ otherwise. 

\begin{figure}[!t]
\centering
 \includegraphics[width=\textwidth]{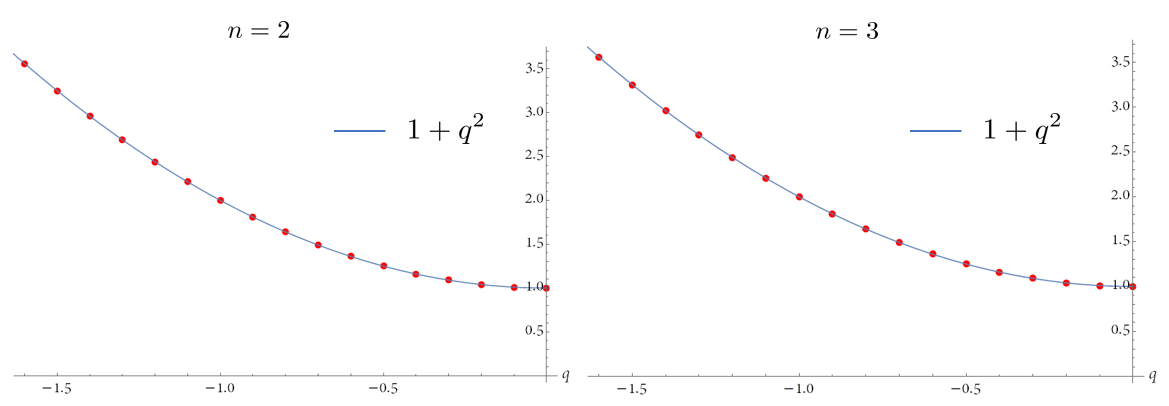}
\caption{Numerical optimization of the bound for negative $q$ ($n=2,3$): we numerically maximised $\f{ [A,B]_q}^2/\left(\f{A}^2 \f{B}^2\right)$ for traceless $A$. It is seen that the maximum (red points) lies on the curve $1+q^2$ (solid line). }\label{fig:Tl23neg}
\end{figure}

\begin{figure}[!t]
\centering
 \includegraphics[width=0.6\textwidth]{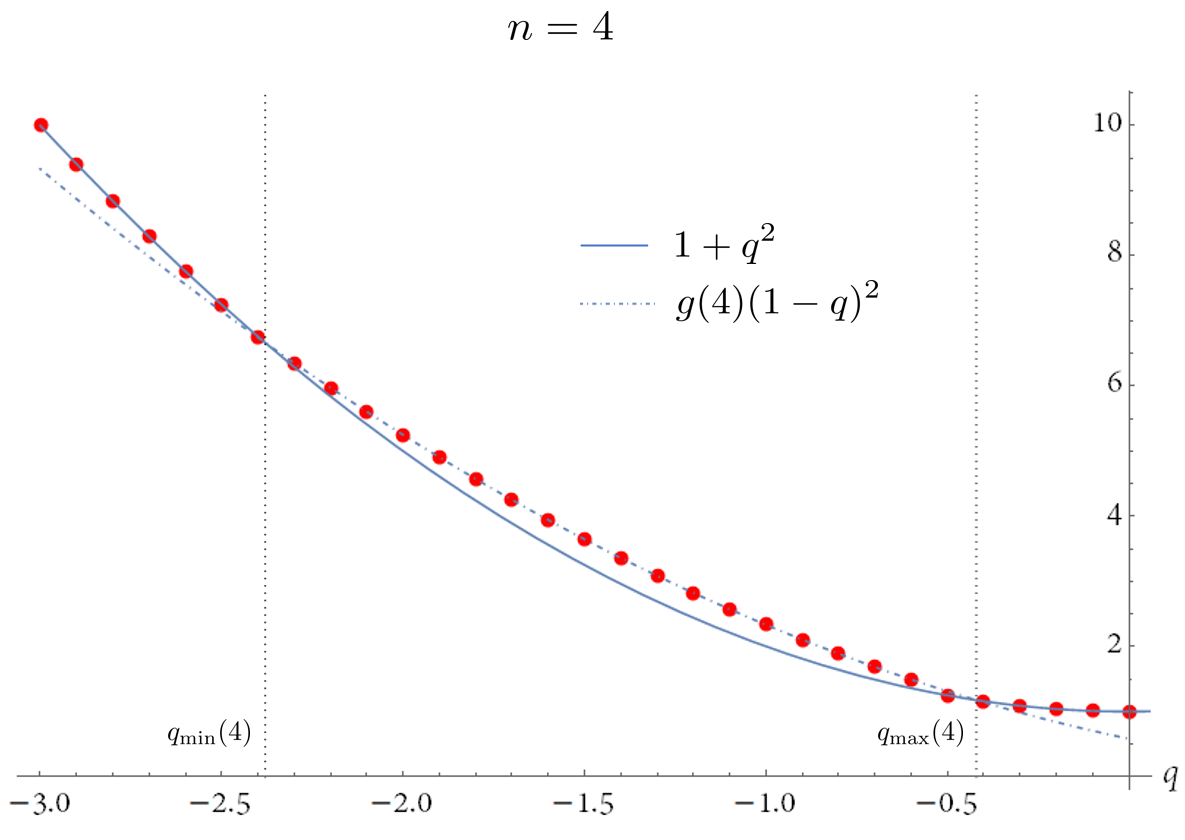}
\caption{Numerical optimization of the bound for negative $q$ ($n=4$): we numerically maximised $\f{ [A,B]_q}^2/\left(\f{A}^2 \f{B}^2\right)$ for traceless $A$. The maximum (red points) lies on the curve $g(4)(1-q)^2$ if $q \in [q_{\min},q_{\max}]$ and on the curve $1+q^2$ otherwise, which contrasts with the case when $n=2,3$.
}\label{fig:Tl4neg}
\end{figure}

\section{Conclusion}

In this paper, we considered the generalization of BW inequality with $q$-deformed commutator in two scenarios. 
In the first scenario, the bound for general matrices is investigated: for $q<0$, the sharp bound is trivially given by \eqref{eq:Neg}. 
For positive $q$, the sharp bound with normal matrices is given by \eqref{eq:q_commutator_BW}. 
However, it seems to be far from trivial to give an upper bound for the general case.  
In the second scenario, we study the bound under the constraint that one of the matrices is traceless, and this constraint is founded on the idea that the trace of the participating matrices is irrelevant in the BW-inequality. For both positive and negative $q$, we give conjectures \ref{con:TLpos} and \ref{CON:Trnq} which are numerically supported and are proved for $n=2$. 
\section*{Acknowledgements}
D.C. was supported by the Polish National Science Centre Project No. 2018/30/A/ST2/00837. G. K. was supported in part by JSPS KAKENHI Grants No. 17K18107.

\FloatBarrier

\end{document}